\documentclass{article}
\usepackage{makeidx}
\usepackage{amsmath,amsfonts,graphicx} 
\usepackage{amsthm}
\usepackage{multirow}
\usepackage{multirow}
\usepackage{mathtools}

\usepackage{float}

\usepackage{algorithmic}    
\usepackage[ruled, vlined, longend, linesnumbered]{algorithm2e}

\def\smallddots{\mathinner{\raise7pt\hbox{.}\raise4pt\hbox{.}\raise1pt\hbox{.}}}
\def\smallsdots{\mathinner{\raise1pt\hbox{.}\raise4pt\hbox{.}\raise7pt\hbox{.}}}

\DeclareMathOperator{\diag}{diag} 

\DeclareMathOperator{\rank}{rank}
\DeclareMathOperator{\nrank}{nrank}

\DeclareMathOperator*{\argmax}{arg\,max}   
\DeclareMathOperator*{\vol}{{\it v}_2}
\DeclareMathOperator*{\volk}{{\it v}_{2, r}}
\DeclarePairedDelimiter\parens{\lparen}{\rparen}
\newcommand{\volp}[1]{\vol\parens*{#1}} 

\numberwithin{equation}{section}
\numberwithin{table}{section}
\newtheorem{theorem}{Theorem}[section]
\newtheorem{lemma}{Lemma}[section]

\newtheorem{corollary}{Corollary}[section]

\newtheorem{example}{Example}[section]
\newtheorem{definition}{Definition}[section]

\newtheorem{remark}{Remark}[section]

\newtheorem{fig.}{FIGURE}[section]

\begin{document}
\author{}

\title{\bf CUR Low  Rank Approximation at Deterministic Sublinear Cost\footnote{Some results of this paper have been presented  in August of 2019 at  MMMA, Moscow, Russia, and at MACIS 2019, on November 13--15, 2019, in Gebze, Turkey. Our extension of Fast Multipole Method to Superfast Multipole
Method was presented at the SIAM Conference on Computational Science and Engineering
in Atlanta, Georgia, on February-March 2017.}}
\smallskip

\author{Qi Luan$^{[1],[a]}$,
Victor Y. Pan$^{[1,2],[b]}$, and 
John Svadlenka$^{[1],[c]}$
\and\\
$^{[1]}$ Ph.D. Programs in  Computer Science and Mathematics \\
The Graduate Center of the City University of New York \\
New York, NY 10036 USA 
\\
$^{[2]}$ Department of Computer Science \\
Lehman College of the City University of New York \\
Bronx, NY 10468 USA 
\\
$^{[a]}$ qi\_luan@yahoo.com \\ 
$^{[b]}$ victor.pan@lehman.cuny.edu \\ 
homepage: http://comet.lehman.cuny.edu/vpan/  \\
$^{[c]}$ jsvadlenka@gradcenter.cuny.edu \\
}
 
\date{}



%

 
\maketitle
\begin{abstract}
A matrix algorithm runs at {\em sublinear cost} if it uses much fewer memory cells and arithmetic operations  than the input matrix has entries. Such algorithms are indispensable for Big Data Mining and Analysis. Quite typically in that area the input matrices are so immense that realistically one can only access a small fraction of  all their entries but can access and process at sublinear cost their Low Rank Approximation {\em (LRA)}. Can, however, we compute LRA at sublinear cost? Adversary argument shows that the output of any algorithm running at sublinear cost is  extremely far from LRA of the worst case input matrices and  even of the matrices of small families of our Appendix, but we prove that some deterministic  sublinear cost algorithms output  reasonably close LRA in a memory efficient form of CUR LRA if an input matrix admits LRA and is Symmetric Positive Semidefinite or is  very close to a low rank  matrix. The latter result is technically simple but provides some (very limited but long overdue) support for the well-known empirical efficiency of sublinear cost LRA by means of Cross-Approximation. We demonstrate the power of application of such LRA by turning the Fast Multipole celebrated Method into  Superfast Multipole  Method. The design  and analysis of our algorithms rely on extensive prior study of the link of  LRA of a matrix to maximization of its volume. 
   
   
\end{abstract}
 
 



\section{Introduction}\label{sintro}


{\bf 1.1. LRA problem.} 
An $m\times n$ matrix $W$ admits its close approximation of rank at most $r$ if and only if
 the matrix $W$ satisfies the bound 
\begin{equation}\label{eqlrk}  
W=AB+E,~||E||\le \epsilon~ ||W||,
\end{equation}
for $A\in \mathbb C^{m\times r}$,
 $B\in \mathbb C^{r\times n}$, 
  $(m+n)r\ll mn$,
  a~fixed matrix~norm $||\cdot||$, and~a~fixed  tolerance $\epsilon$, chosen in context, e.g., depending on computer precision, the assumptions about an input, and the requirements to
   the output.\footnote{Here and hereafter $a\ll b$ and $b\gg a$ mean that the ratio $|a/b|$ is small in context. It is customary in this field to rely on the informal basic concept of low rank approximation; in high level description of our LRA algorithms we also use some other informal concepts such as “large”, “small”, “tiny”, “typically”, “near”, “close”, and “approximate”, quantified in context; we complement this description with formal presentation and analysis.}
Such an LRA approximates the $mn$  entries of $W$ by using $(m+n)r$ entries of $A$ and $B$, and one  can compute the approximation  $AB{\bf v}$ of the product $W{\bf v}$ for a vector
${\bf v}$ by applying less  than $2(m+n)r\ll mn$ arithmetic operations.
This is a crucial benefit 
in applications of LRA to Big 
Data Mining and Analysis, where quite typically the size $m\times n$ of an  input matrix is so immense that  one can only access a tiny fraction of all its  
$mn$  entries but where
 the matrix admits LRA.   
 
 Can we, however, compute close  LRA 
{\em at sublinear cost}, that is, by using much fewer memory cells and  arithmetic operations than the number $mn$ of all  entries of an input matrix? 
Based on adversary argument one can prove that the output of any algorithm running at sublinear cost is
 extremely far from LRA   of the worst case inputs and even of the matrices of small families of our Appendix, but
 for two decades   
Cross--Approximation (C–A) 
iterations (see Sections 1.3 and \ref{scait}) have been routinely  computing LRA worldwide at 
sublinear cost, which further decreased where an input matrix is sparse. Moreover the 
      iterations output LRA 
in its special form of CUR LRA (see Section \ref{scur}), particularly memory efficient. 

It may be surprising, but no formal support has been 
provided for this empirical efficiency so far, although it is well-understood by now that the power of the algorithm is closely linked   to computing a submatrix of an input matrix that has the  maximal or a nearly maximal volume among all submatrices of a fixed size (see Section \ref{svlm1}). Having such a submatrix one can readily define a CUR LRA with output error norm  within roughly a factor of
 $\sqrt {mn}$ from its optimal bound,
 defined by the Eckart-Young-Mirski theorem. 

Furthermore for a large class of input matrices $W$  any LRA  which is
close within a reasonable factor
to $W$ is a springboard 
for the sublinear cost algorithm for iterative refinement 
 recently proposed and  tested in 
 \cite{LPSa}.  
 

{\bf 1.2. Our  results.}
Our main result is  
 sublinear cost deterministic algorithms whose output  CUR LRAs 
 approximate  any $n\times n$  Symmetric Positive Semidefinite {(\em SPSD)} matrix admitting LRA. The relative output error of these algorithms   is
within a factor of $(1+\epsilon)n$ from the  optimal bound
for any fixed positive tolerance  
$\epsilon$ and for any SPSD input
provided that that $n/r$ is in  
$O(1/\epsilon)$.
This is in sharp contrast to the case of general input where at sublinear cost one cannot verify even whether
 a candidate LRA is closer than just a matrix filled with 0s to some matrix 
 of our family of the Appendix.
The only known alternative to our  deterministic SPSP algorithm is the randomized algorithm  of \cite{MW17}, which
relies on completely different techniques
and is harder to implement because it is much more involved; to its theoretical advantage, however,  its  output 
LRA  becomes nearly optimal 
for 
very large $n$.

In Part III we readily elaborate upon a simple observation 
that a single loop of C-A iterations computes a $p\times r$ or $r\times q$ submatrix having maximal volume within a bounded factor, provided that an input matrix $W$ has rank $r\le \min\{p,q\}$.
It follows that a single C-A loop
outputs a reasonably close CUR LRA
also in  a small neighborhood of $W$.
Clearly this support of C-A iterations is very limited but is the first step towards meeting a long overdue challenge.

In Sections \ref{sextpr} and  \ref{sHSS} we demonstrate the power of C-A  iterations by applying them to the acceleration of the {\em Fast Multipole Method (FMM)},  which is one of ``the Top Ten Algorithms of the 20-th century" \cite{C00}. The method involves LRAs of the off-diagonal blocks of an input matrix.
 In some applications of FMM such LRAs are a part of an input, but in other  important cases they are not known, and then their computation becomes bottleneck of the method. In these cases incorporation of sublinear cost  C-A iterations implies dramatic acceleration 
 of FMM, thus defining {\em Superfast Multipole Method}.  We demonstrate this progress with extensive 
 tests on the real world data from \cite{XXG12}. 

Our work can provide some new insight into the link of LRA to the maximization of matrix volume, extensively studied in  the LRA field. 

   
   

{\bf 1.3. Related works.}
 For decades
the researchers working in  Numerical Linear Algebra, later joined by those from Computer and Data Sciences, have been
contributing  extensively to LRA,
whose immense bibliography  
can be  traced back to \cite{G65}, \cite{BG65}, 
  \cite{K85},  \cite{C87},  \cite{CH90},  \cite{CH92},
   \cite{HLY92},  \cite{HP92},  \cite{CI94} and can be  partly accessed  via 
\cite{M11}, \cite{HMT11}, \cite{W14}, \cite{CLO16}, \cite{KS17}, \cite{B19}, \cite{M18},
\cite{IVWW19}, \cite{M19}, \cite{TYUC19},  \cite{BCW20}, and \cite{YXY20}.

   Cross-Approximation (C-A) iterations for LRA have specialized to LRA the Alternating Directions Implicit (ADI) method  \cite{PTVF07} and are more commonly referred to as Adaptive Cross-Approximation (ACA) iterations.
     Their derivation  can be traced to
      \cite{GZT95},  \cite{T96}, \cite{GTZ97},   \cite{GZT97},
     \cite{T00}, and \cite{B00} and was  strengthened with exploiting 
      the maximization of matrix volume 
        in \cite{BR03}, 
        \cite{BG06}, \cite{GOSTZ10},   \cite{B11},
      \cite{GT11},   and \cite{OZ18}.           
On matrix volume and its link to LRA see
    also \cite{K85},  \cite{B-I92},
      \cite{CI94},  \cite{GT01}, and \cite{GT11}.
     
  Our  formal theoretical support for  empirical efficiency of 
  C-A   
     iterations simplifies   
  and strengthens the results of 
  \cite{PLSZ16}, \cite{P16}, and 
\cite[Part II]{PLSZ19}, where  much more involved 
 proofs implied similar results but with larger upper bounds on the output errors. This was because the latter works
 relied on the theorems stated in \cite{OZ18}
 in  weaker form than they are actually implied by their
 advanced proofs in \cite{OZ18}. 
   
 The papers \cite{PLSZ16}, \cite{P16}, and 
\cite{PLSZ19} extend to LRA the   techniques of randomized matrix computations 
 in \cite{PQY15}, \cite{PZ17a}, and \cite{PZ17b}. In particular the latter papers  avoid the communication extensive stage of pivoting  in Gaussian elimination by means of  pre-processing of an input matrix. 
   
The algorithms of these earlier papers compute at sublinear cost close LRA of a large class of matrices that admit LRA
and in a sense even  of most  of
such matrices.  They refer to  
 the algorithms as
{\em superfast} rather than sublinear cost  
 algorithms, as this is common  in the study of computations with matrices having small  displacement rank, 
 including Toeplitz, Hankel, Vandermonde, Cauchy, Sylvester, Frobenius, resultant, and Nevanlinna-Pick  matrices
  (see \cite{KKM79}, \cite{OP98},  \cite{P00},  \cite{P01}, \cite{P15}, and the bibliography therein).   

Our  deterministic algorithms in Parts II and III have departed from the recent progress in  \cite{OZ18} and \cite{CKM19} and are  technically independent of the  important randomized sublinear cost LRA for special classes of input matrices devised in
 the papers  \cite{MW17}, \cite{IVWW19},  \cite{SW19},   
  \cite{BCW20}.
  
  
{\bf 1.4. Organization of the paper.}
Part I of our paper is made up of the next four sections: we cover
background material on matrix computations in the next section, 
 CUR LRA in Section \ref{scur},   matrix volumes and the impact of their maximization
on LRA in Section \ref  {svlm1},
 and C-A iterations in Section \ref{scait}.

 Sections \ref{smnthr}--\ref{sect:comp_anal} make up Part II
 of the paper.
 In Section \ref{smnthr} we state our main results for SPSD inputs. We present  supporting algorithms  and prove their correctness  in Sections \ref{sect:proof_thm1} and
 \ref{sect:proof_thm2}. We estimate their complexity in Section \ref{sect:comp_anal}.

Part III of the paper is made up of 
Sections \ref{svlmxmz} -- \ref{sHSS}.
 In  Sections \ref{svlmxmz} -- \ref{scurac} we  prove that already a single C-A loop outputs  
a reasonably close  LRA of a matrix that admits a very close LRA. In Section \ref{sextpr} we  apply sublinear cost LRA to the acceleration of FMM and in Section
  \ref{sHSS} support this with extensive test results, which are the contribution of the third author of this paper.

 
 In the Appendix  
 we  specify small  families of matrices that admit LRA but do not allow its accurate computation by means of any    sublinear cost algorithm.


\medskip
\medskip
\medskip

{\bf \large PART I. BACKGROUND MATERIAL}

  
\section{Basic definitions for matrix com\-putations, Eckart-Young-Mirsky theorem, and a lemma}\label{sgnm} 
 

$\mathbb R^{m\times n}$ is the class of $m\times n$ 
matrices with 
real entries.\footnote{We simplify  our presentation by confining it to computations  with real matrices.}

$I_s$ denotes the $s\times s$ identity matrix. $O_{q,s}$ denotes the $q\times s$  matrix filled with zeros.

$\diag (B_1,\dots,B_k)=\diag(B_j)_{j=1}^k$  denotes a $k\times k$ block diagonal matrix   
with diagonal blocks $B_1,\dots,B_k$.  
  
$(B_1~|~\dots~|~B_k)$ and $(B_1,\dots,B_k)$ 
denote a $1\times k$ block matrix with blocks $B_1,\dots,B_k$.  

$W^T$ and $W^*$ denote the transpose and the Hermitian transpose 
of an $m\times n$ matrix $W=(w_{ij})_{i,j=1}^{m,n}$, respectively.
 $W^*=W^T$ if the  matrix $W$ is real.
 
For two sets $\mathcal I\subseteq\{1,\dots,m\}$  
and $\mathcal J\subseteq\{1,\dots,n\}$   define
the submatrices
\begin{equation}\label{eq111}
W_{\mathcal I,:}:=(w_{i,j})_{i\in \mathcal I; j=1,\dots, n},  
W_{:,\mathcal J}:=(w_{i,j})_{i=1,\dots, m;j\in \mathcal J},~ 
W_{\mathcal I,\mathcal J}:=(w_{i,j})_{i\in \mathcal I;j\in \mathcal J}.
\end{equation} 

An $m\times n$ matrix $W$ is {\em orthogonal} 
if $W^*W=I_n$ or $WW^*=I_m$.

{\em Compact SVD} of a matrix $W$, hereafter just {\em SVD}, is defined by the equations 
\begin{equation}\label{eqsvd}
W=S_W\Sigma_WT_W^*,
\end{equation}  
where
$$S_W^*S_W=
T_W^*T_W=I_{\rho},~ 
\Sigma_W:=\diag(\sigma_j(W))_{j=1}^{\rho},~\rho={\rank (W)},$$
$\sigma_j(W)$
denotes the $j$th largest singular value of $W$ for $j=1,\dots,\rho$;
 $\sigma_{\rho}(W)>\sigma_j(W)=0~{\rm for}~j>\rho$.
 
  $||W||=||W||_2$, $||W||_F$, and $||W||_C$
 denote 
spectral, Frobenius, and Chebyshev norms
of a matrix $W$, respectively,
such that (see \cite[Section 2.3.2 and Corollary 2.3.2]{GL13})
$$||W||=\sigma_1(W),~ ||W||_F^2:=\sum_{i,j=1}^{m,n}|w_{ij}|^2=\sum_{j=1}^{\rank(W)}\sigma_j^2(W),~ 
||W||_C:=\max_{i,j=1}^{m,n}|w_{ij}|,$$  
\begin{equation}\label{eq0}
||W||_C\le ||W||\le ||W||_F\le \sqrt {mn}~||W||_C,~
 ||W||_F^2\le
 \min\{m,n\}~||W||^2.
\end{equation}

 By virtue of the Eckart-Young-Mirsky theorem below, one can obtain optimal rank-$\rho$ approximation of a matrix  under both spectral and Frobenius norms by setting
 to 0 all its singular values but its 
 $\rho$ largest ones:

\begin{theorem}\label{thtrnc} {\rm \cite[Theorem 2.4.8]{GL13}.} The optimal
bound on the error of rank-$\rho$ approximation of a matrix $M$ 
 is equal to
 $\sigma_{\rho+1}(M)$ under the spectral norm $|\cdot|=||\cdot||$
and to 
$(\sum_{j\ge \rho}\sigma_j^2(M))^{1/2}$ 
under the Frobenius norm $|\cdot|=||\cdot||_F$.
\end{theorem}

$W^+:=T_W\Sigma_W^{-1}S_W^*$ is the Moore--Penrose 
pseudo inverse of an $m\times n$ matrix $W$.
\begin{equation}\label{eqsgm}
||W^+||\sigma_{r}(W)=1
\end{equation}
for a rank-$r$ matrix $W$.

A matrix $W$  has 
$\epsilon$-{\em rank} at most $r>0$
for a fixed tolerance $\epsilon>0$ if there is a matrix $W'$  
 of rank $r$ such that 
 $||W'-W||/||W||\le \epsilon$.  
We write $\nrank(W)=r$ and say that a
matrix $W$ has {\em numerical rank} $r$
if it has $\epsilon$-rank $r$
for a small tolerance  $\epsilon$ fixed in context and typically being linked to
machine precision or the level of relative errors of the computations (cf. \cite[page 276]{GL13}).

 
 \begin{lemma}\label{lehg}
Let $G\in\mathbb R^{k\times r}$, 
$\Sigma \in\mathbb R^{r\times r}$ and
$H\in\mathbb R^{r\times l}$
and let  the matrices $G$, $H$ and 
$\Sigma$ have full rank 
$r\le \min\{k,l\}$.
Then 
$||(G\Sigma H)^+||
\le ||G^+||~||\Sigma^+||~||H^+||$. 
\end{lemma}
For the sake of completeness of our
presentation we include a proof of this well-known result. 
\begin{proof} 
Let $G=S_G\Sigma_GT_G$ and $H=S_H\Sigma_HT_H$ be SVDs
where $S_G$, $T_G$, $D_H$, and $T_H$ are orthogonal matrices,
$\Sigma_G$ and $\Sigma_H$ are the $r\times r$ 
nonsingular diagonal matrices of the singular values,
and $T_G$ and $S_H$ are  $r\times r$ matrices. 
Write $$M:=\Sigma_GT_G\Sigma S_H\Sigma_H.$$ 
Then 
$$M^{-1}=\Sigma_H^{-1} S_H^*\Sigma^{-1}T_G^*\Sigma_G^{-1},$$ and consequently
$$||M^{-1}||\le ||\Sigma_H^{-1}||~|| S_H^*||~||\Sigma^{-1}||~||T_G^*||~||\Sigma_G^{-1}||.$$ 
Hence
$$||M^{-1}||\le||\Sigma_H^{-1}||~
||\Sigma^{-1}||~|\Sigma_G^{-1}||$$ because $S_H$ and $T_G$
are orthogonal matrices. It follows from 
(\ref{eqsgm}) for $W=M$ that

$$\sigma_r(M)\ge\sigma_r(G)\sigma_r(\Sigma)\sigma_r(H).$$

Now let  $M=S_M\Sigma_MT_M$ be SVD
where $S_M$ and $T_M$ are $r\times r$ orthogonal matrices. 

Then $S:=S_GS_M$ and $T:=T_MT_H$ are orthogonal matrices,
and so $G\Sigma H=S\Sigma_M T$ is SVD. 

Therefore 
$\sigma_r(G\Sigma H)=\sigma_r(M)\ge  \sigma_r(G)\sigma_r(\Sigma)\sigma_r(H)$.
Combine this bound with (\ref{eqsgm})
for $W$ standing for $G$, $\Sigma$,
$H$, and $G \Sigma H$.
\end{proof}
            

\section{CUR LRA}\label{scur} 

 

{\em CUR LRA} of a matrix $W$ of numerical rank at most $r$ is defined by three matrices 
$C$, $U$, and $R$, with $C$ and 
$R$  made up of $l$ columns
and  $k$ rows of $W$, respectively,
$U\in \mathbb R^{l\times k}$  said to be the {\em nucleus} of 
CUR LRA, 
  \begin{equation}\label{eqklr}  
0<r\le k\le m,~r\le l\le n,~kl\ll mn,
\end{equation}
 \begin{equation}\label{eqcurlra}   
W=CUR+E,~{\rm and}~||E||/||W||\le \epsilon,~{\rm for~a~small~ tolerance}~\epsilon>0.
\end{equation}
CUR LRA is a special case
 of LRA of (\ref{eqlrk}) where $k=l=r$ and, say, $A=LU$, $B=R$.
 Conversely, given LRA of (\ref{eqlrk})
  one can compute CUR LRA of (\ref{eqcurlra}) at sublinear cost
  (see \cite{LPa} and  \cite{PLSZ19}).
  
Define a {\em canonical} CUR LRA as follows. 

(i) Fix two sets of  columns and rows of $W$
and define its two submatrices $C$ and $R$
made up of these columns and rows, respectively.
 
 (ii) Define the $k\times l$  submatrix 
 $W_{k,l}$
made up of all common entries of $C$ and $R$, and call it a {\em CUR generator}. 

(iii) Compute its rank-$r$ truncation $W_{k,l,r}$ by setting to 0 all its
  singular values, 
 except for the
 $r$ largest ones. 

(iv)  Compute the Moore--Penrose pseudo inverse
 $U=:W_{k,l,r}^+$ and call it the {\em nucleus} of CUR LRA
 of the matrix $W$ (cf. \cite{DMM08}, \cite{OZ18});  see an alternative choice of a nucleus in \cite{MD09}).

$W_{r,r}=W_{r,r,r}$,
 and if a CUR generator  
 $W_{r,r}$ is  nonsingular,
 then $U=W_{r,r}^{-1}$.
 

\section{Matrix Volumes}\label{svlm1}
  

\subsection{Definitions and Hadamard's bound}\label{svlmdf}
  


 \begin{definition}\label{defmxv}
For   
three integers $k$,  $l$, and $r$ such that $1\le r\le \min\{k,l\}$,
define the {\em volume} $v_2(M):=\prod_{j=1}^{\min\{k,l\}} \sigma_j(M)$ and  $r$-{\em projective volume}
$v_{2,r}(M):=\prod_{j=1}^r \sigma_j(M)$
of a  $k\times l$ matrix $M$ such that
$v_{2,r}(M)=v_{2}(M)~{\rm if}~r=\min\{k,l\}$,
$v_2(M)=\sqrt{\det(MM^*)}$ if $k\le l$;
$v_2(M)=\sqrt{\det(M^*M)}$ if $k\ge l$,
$v_2(M)=|\det(M)|$ if $k = l$.
\end{definition}


\begin{definition}
The volume of a $k\times l$ submatrix $W_{\mathcal I,\mathcal J}$
of a matrix $W$ is $h$-{\em maximal}  
over all $k\times l$ submatrices if
it is maximal up to a factor of $h$ for $h>1$.
The volume of a submatrix $W_{\mathcal I,\mathcal J}$ is 
{\em locally} $h$-maximal if it is $h$-maximal over all $k\times l$ submatrices of $W$
that differ from the  submatrix $W_{\mathcal I,\mathcal J}$ by 
a single column and/or a single row.
\end{definition} 
 
 \begin{definition}\label{defmxv1}  
The volume of $W_{\mathcal I,\mathcal J}$ is {\em column-wise} (resp.  
{\em row-wise}) $h$-maximal if it is $h$-maximal in the submatrix 
$W_{\mathcal I,:}$ (resp. $W_{:,\mathcal J})$. 
The volume of a submatrix $W_{\mathcal I,\mathcal J}$ is 
 {\em column-wise} (resp. 
{\em row-wise}) {\em locally} $h$-maximal if it is $h$-maximal over all submatrices of $W$
that differ from the  submatrix $W_{\mathcal I,\mathcal J}$ by 
a single column  (resp. single row).
Call volume  $(h_c,h_r)$-{\em maximal}
 if it is  both
 column-wise $h_c$-maximal and row-wise $h_r$-maximal.
Likewise  define {\em locally}  $(h_c,h_r)$-{\em maximal} volume.
Write {\em maximal} instead of $1$-maximal and $(1,1)$-maximal in these definitions. 
Extend all of them to $r$-projective volumes.
\end{definition}

 For a $k\times l$ matrix  $M=(m_{ij})_{i,j=1, 1}^{k, l}$
write ${\bf m}_j:=(m_{ij})_{i=1}^k$ and  ${\bf\bar m}_i:=((m_{ij})_{j=1}^l)^*$ for all $i$ and $j$. For $k=l=r$
 recall  {\em Hadamard's bound}
 
$v_2(M)=|\det(M)|\le \min~\{\prod_{j=1}^r||{\bf m}_j||,~
\prod_{i=1}^r||{\bf \bar m^*}_j||,~r^{r/2}\max_{i,j=1}^r |m_{ij}|^r\}.$

 
\subsection{The impact of volume maximization on CUR LRA}\label{simpvlmxmz1}
 
 
The estimates of the two following theorems in the Chebyshev matrix norm $||\cdot||_C$ 
increased by  a factor of
 $\sqrt {mn}$ turn into  estimates in the Frobenius  norm $||\cdot||_F$ 
 (see (\ref{eq0})).

\begin{theorem}\label{th12}   \cite{OZ18}.\footnote{The theorem first appeared in \cite[Corollary 2.3]{GT01} in the  special case where $k=l=r$ and $m=n$.}
Suppose that $r:=\min\{k,l\}$, 
 $W_{\mathcal I,\mathcal J}$
is the $k\times l$ CUR generator, 
$U=W_{\mathcal I,\mathcal J}^+$  is the
nucleus  of
a canonical
 CUR LRA of an $m\times n$ matrix $W$, $E=W-CUR$,  $h\ge 1$,
 and the volume of
$W_{\mathcal I,\mathcal J}$ is locally  $h$-maximal, that is,  
  $$h~v_{2}(W_{\mathcal I,\mathcal J})=\max_B v_2(B)$$
where   the maximum is over all $k\times l$  submatrices $B$
 of the matrix $W$ 
 that differ from $W_{\mathcal I,\mathcal J}$ in at most one row and/or column. 
Then  
$$||E||_C\le h~f(k,l)~\sigma_{r+1}(W)~~{\rm for}~~
f(k,l):=\sqrt{\frac{(k+1)(l+1)}{|l-k|+1}}.
$$ 
\end{theorem}     

\begin{theorem}\label{th3}  \cite{OZ18}.
Suppose that  
 $W_{k,l}=W_{\mathcal I,\mathcal J}$
is a $k\times l$ submatrix of 
 an $m\times n$ matrix $W$,
 $U=W_{k,l,r}^+$ 
 is the nucleus of
a canonical CUR LRA of $W$, $E=W-CUR$, $h\ge 1$,
 and the $r$-projective volume of
$W_{\mathcal I,\mathcal J}$ is  locally  $h$-maximal, that is, 
  $$h~v_{2,r}(W_{\mathcal I,\mathcal J})=\max_B v_{2,r}(B)$$                                                                                                                                                                                                                                                                                                                                                                                                                                                                                                                             
 where   the maximum is over all $k\times l$ submatrices $B$
 of the matrix $W$  that differ from $W_{\mathcal I,\mathcal J}$ in at most one row and/or column.  
Then 
$$||E||_C\le h~f(k,l,r)~\sigma_{r+1}(W)
~~{\rm for}~~f(k,l,r):=
\sqrt{\frac{(k+1)(l+1)}{(k-r+1)(l-r+1)}}. 
$$
\end{theorem}  



 \begin{remark}\label{re3}
Theorems \ref{th12} and \ref{th3} have been stated in  \cite{OZ18} under assumptions  
that the matrix $W_{\mathcal I,\mathcal J}$ has (globally) $h$-maximal volume or 
 $r$-projective volume, respectively, but their proofs in \cite{OZ18} support the 
  above extensions to the case of locally maximal volume and $r$-projective volume. 
\end{remark}

 
\subsection{The volume and $r$-projective volume of a perturbed matrix}\label{svlmprtrb}

   
\begin{theorem}\label{thprtrbvlm} 
Suppose that 
$W'$ and $E$ are $k\times l$ 
matrices, 
$\rank(W')=r\le\min\{k,l\}$,
$W=W'+E$, and   
$||E||\le \epsilon$. Then 
\begin{equation}\label{eqrts}
\Big (1-\frac{\epsilon}{\sigma_r(W)}\Big )^{r}
\le \prod_{j=1}^r\Big (1-\frac{\epsilon}{\sigma_j(W)}\Big ) 
\le\frac{v_{2,r}(W)}{v_{2,r}(W')}\le
 \prod_{j=1}^r\Big (1+\frac{\epsilon}{\sigma_j(W)}\Big )\le   
\Big (1+\frac{\epsilon}{\sigma_r(W)}\Big )^r. 
\end{equation}    
If $\min\{k,l\}=r$, then $v_2(W)=v_{2,r}(W)$,  $v_2(W')=v_{2,r}(W')$, and 

\begin{equation}\label{eqrts1}
\Big (1-\frac{\epsilon}{\sigma_r(W)}\Big )^{r}
\le\frac{v_2(W)}{v_2(W')}=
\frac{v_{2,r}(W)}{v_{2,r}(W')}\le
\Big (1+\frac{\epsilon}{\sigma_r(W)}\Big )^r.
\end{equation}
\end{theorem}
\begin{proof} 
Bounds (\ref{eqrts}) follow because a perturbation of a matrix within 
a norm bound $\epsilon$ changes its singular values by at most $\epsilon$ (see 
\cite[Corollary 8.6.2]{GL13}).
Bounds (\ref{eqrts1}) follow because
$v_2(M)=v_{2,r}(M)=\prod_{j=1}^r\sigma_j(M)$
for any $k\times l$ matrix $M$ with $\min\{k,l\}=r$, in particular 
for $M=W'$ and $M=W=W'+E$.
\end{proof}

If the ratio 
$\frac{\epsilon}{\sigma_r(W)}$ is small,
then $\Big (1-\frac{\epsilon}{\sigma_r(W)}\Big )^{r}=1-O\Big (\frac{r\epsilon}{\sigma_r(W)}\Big )$ and  
$\Big (1+\frac{\epsilon}{\sigma_r(W)}\Big )^r=
1+O\Big (\frac{r\epsilon}{\sigma_r(W)}\Big )$,
which shows that the relative perturbation of
the  volume is amplified by at most a factor of $r$ in comparison to 
the relative perturbation of the $r$ largest singular values.

 
\subsection{The volume and $r$-projective volume of a matrix product}\label{svlmprd}


\begin{theorem}\label{thvolfctrg} 
{\rm [Examples \ref{ex1} and  \ref{ex2} below show some limitations on the extension of this theorem.]}
Suppose that $W=GH$ for an $m\times q$ matrix $G$ and a $q\times n$
matrix $H$.
Then 

(i) $v_2(W)=v_2(G)v_2(H)$
if $q=\min\{m,n\}$; 
$v_2(W)=0\le v_2(G)v_2(H)$ if 
$q<\min\{m,n\}$.

(ii) $v_{2,r}(W)\le v_{2,r}(G)v_{2,r}(H)$ for $1\le r\le q$, 

 (iii) $v_2(W)\le v_2(G)v_2(H)$ if $m=n\le q$.
\end{theorem}
The theorem is proved in \cite{OZ18}.
We present an alternative proof.
\begin{example}\label{ex1}
If $G$ and $H$ are orthogonal matrices and if $GH=O$, 
then $v_2(G)=v_2(H)=v_{2,r}(G)=v_{2,r}(H)=1$ and 
$v_2(GH)=v_{2,r}(GH)=0$ for all $r\le q$.
\end{example}

\begin{example}\label{ex2}
If $G=(1~|~0)$ and $H=\diag(1,0)$, 
then $v_2(G)=v_2(GH)=1$ and 
$v_2(H)=0$.
\end{example}

\begin{proof}

First prove claim (i). 

Let $G=S_G\Sigma_GT^*_G$ and $H=S_H\Sigma_HT^*_H$ be SVDs
such that $\Sigma_G$, $T^*_G$, $S_H$, $\Sigma_H$,  and $U=T^*_GS_H$ 
are $q\times q$ matrices and
$S_G$, $T^*_G$, $S_H$, $T^*_H$,  and $U$ are orthogonal matrices.

Write $V:=\Sigma_GU\Sigma_H$. Notice that 
$\det (V)=\det(\Sigma_G)\det(U)\det(\Sigma_H)$. Furthermore
 $|\det(U)|=1$  because 
$U$ is a square orthogonal matrix.
Hence $v_2(V)=|\det (V)|=|\det(\Sigma_G)\det(\Sigma_H)|=v_2(G)v_2(H)$.

Now let $V=S_V\Sigma_VT_V^*$ be SVD where $S_V$, $\Sigma_V$, and $T_V^*$ 
are $q\times q$ matrices and where $S_V$ and $T_V^*$ are orthogonal matrices.

Observe that $W=S_GVT^*_H=S_GS_V\Sigma_VT_V^*T^*_H
=S_W \Sigma_VT^*_W$ where $S_W=S_GS_V$ and $T^*_W=T_V^*T^*_H$
are orthogonal matrices. 
Consequently
$W=S_W \Sigma_VT^*_W$ is SVD, and so $\Sigma_W=\Sigma_V$.

Therefore $v_2(W)=v_2(V)=v_2(G)v_2(H)$
unless $q<\min\{m,n\}$.
This proves claim (i) because clearly 
$v_2(W)=0$ if $q<\min\{m,n\}$.

Next prove claim (ii).

First assume that $q\le \min\{m,n\}$ as in claim (i)
and let $W=S_W \Sigma_WT^*_W$ be SVD.

In this case we have proven that
 $\Sigma_W=\Sigma_V$ for $V=\Sigma_GU\Sigma_H$, 
$q\times q$ diagonal matrices $\Sigma_G$ and $\Sigma_H$,
and a $q\times q$ orthogonal matrix $U$.
Consequently $v_{2,r}(W)=v_{2,r}(\Sigma_V)$.

In order to prove claim (ii) in the case where
$q\le \min\{m,n\}$, it remains to deduce that
\begin{equation}\label{eqghv}
v_{2,r}(\Sigma_{V})\le v_{2,r}(G) v_{2,r}(H).
\end{equation} 

Notice that
$\Sigma_V=S_V^*VT_V=S_V^*\Sigma_GU\Sigma_HT_V$
for $q\times q$ orthogonal matrices $S_V^*$ and $H_V$.

Let $\Sigma_{r,V}$ denote the $r\times r$
leading submatrix of $\Sigma_V$, and so
$\Sigma_{r,V}=\widehat G \widehat H$
where
$\widehat G:=S_{r,V}^*\Sigma_GU$ and $\widehat H:=\Sigma_HT_{r,V}$ and
where $S_{r,V}$ and $T_{r,V}$ denote the 
$r\times q$ leftmost orthogonal submatrices 
of the matrices $S_V$
and $T_V$, respectively.

Observe that $\sigma_j(\widehat G)\le \sigma_j(G)$ for 
all $j$ because $\widehat G$ is a submatrix of the $q\times q$ matrix $S_{V}^*\Sigma_GU$,
and similarly 
$\sigma_j(\widehat H)\le \sigma_j(H)$ for all $j$. Therefore
$v_{2,r}(\widehat G)=v_{2}(\widehat G)\le v_{2,r}(G)$
and $v_{2,r}(\widehat H)=v_{2}(\widehat H)\le v_{2,r}(H)$.
Also notice that 
$v_{2,r}(\Sigma_{r,V})=v_{2}(\Sigma_{r,V})$.

Furthermore $v_{2}(\Sigma_{r,V})\le 
v_{2}(\widehat G)v_{2}(\widehat H)$
by virtue of claim (i) because 
$\Sigma_{r,V}=\widehat G \widehat H$.

Combine the latter relationships and obtain (\ref{eqghv}), 
which implies claim (ii) in the case where
$q\le \min\{m,n\}$.

Next we extend claim (ii)  to the general case of any  positive 
integer $q$.

Embed a matrix $H$ into a $q\times q$  matrix $H':=(H~|~O)$
banded by zeros if $q>n$. Otherwise write $H':=H$.
 Likewise 
embed a matrix $G$ into a $q\times q$  matrix $G':=(G^T~|~O)^T$
banded by zeros if $q>m$. Otherwise write $G':=G$.

Apply claim (ii) to the $m'\times q$ matrix $G'$ and $q\times n'$ matrix
 $H'$ where $q\le \min\{m',n'\}$. 

Obtain that $v_{2,r}(G'H')\le v_{2,r}(G')v_{2,r}(H')$.

Substitute equations 
$$v_{2,r}(G')=v_{2,r}(G),~v_{2,r}(H')=v_{2,r}(H),~{\rm  and}~ 
v_{2,r}(G'H')=v_{2,r}(GH),$$
which hold because the embedding 
keeps invariant the singular values
and therefore keeps invariant the  volumes  of the matrices 
$G$, $H$,  and $GH$. 

This completes the  proof of claim (ii), which
implies claim (iii) because  
$v_2(V)=v_{2,n}(V)$ if $V$ stands for $G$, $H$, or $GH$
and if $m=n\le q$.
\end{proof}

\section{C--A iterations}\label{scait}
 

C--A iterations recursively apply
two auxiliary Sub-algorithms $\mathcal A$ and $\mathcal B$ (see Algorithm \ref{algdomin}). See the customary recipes for the initialization  of these iterations in
  \cite{B00}, \cite{GOSTZ10}, and  \cite{B11}. 

{\bf Sub-algorithm $\mathcal A$.} Given a 4-tuple of integers $k$, $l$,
$p$, and $q$ such that 
$r\le k\le p$ and $r\le l\le q$, 
Sub-algorithm $\mathcal A$ is applied to  a $p\times q$ matrix and computes its $k\times l$
 submatrix 
whose volume or projective volume 
is maximal up to a fixed factor $h\ge 1$ over all its $k\times l$ submatrices. 
  In Figure \ref{fig4}, reproduced  from \cite{PLSZ19}, we show three C--A iterations in the simple the case where 
$k=l=p=q=r$.
 
 \begin{figure}[ht]
\centering
\includegraphics[scale=0.25]{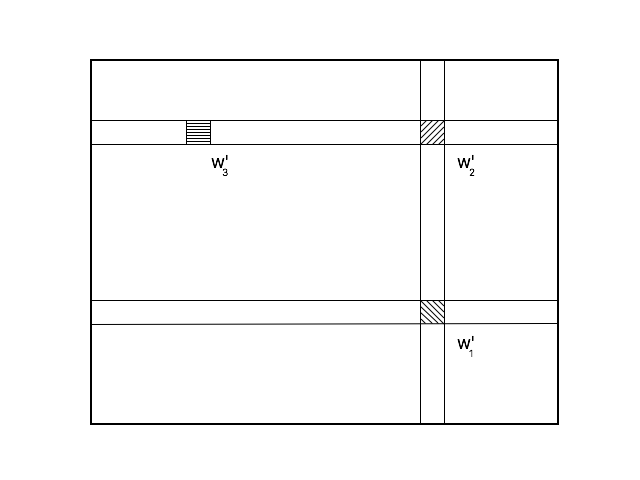}
\caption{The three successive C--A steps output three
striped matrices.}
\label{fig4}
\end{figure}  

{\bf Sub-algorithm $\mathcal B$.} This 
  sub-algorithm verifies whether the  error norm of the CUR LRA built on a fixed CUR                                                                 generator is within a fixed tolerance $\tau$
(see \cite{LPa} for some heuristic recipes for  verification at sublinear cost).


\begin{algorithm}[t]
\caption{C--A Iterations}\label{algdomin}
\begin{quote}
    {\bf Input: } $W\in \mathbb C^{m\times n}$, four positive integers  $r, k, l$, and ITER; a number $\tau>0$.\\
    \smallskip
    {\bf Output: } A CUR LRA of $W$ with an error norm at most $\tau$
    or FAILURE.
    \begin{algorithmic}
    \STATE {\bf Initialization: } Fix a submatrix $W_0$ made up of $l$ columns of $W$ and obtain an initial set $\mathcal{I}_0$.\\
    \smallskip
    {\bf Computations:}\\
    \FOR{$i = 1, 2, \dots, \textrm{ ITER}$}
        \IF{$i$ is even}
            \STATE {\bf "Horizontal"} C--A step:
            \STATE 1. Let $R_i:=W_{\mathcal I_{i-1},:}$ be a $p\times n$  submatrix of $W$.
            \STATE 2. Apply Sub-algorithm $\mathcal{A}$ for $q=n$  to $R_i$ and obtain  a $k\times l$ submatrix $W_{i} = W_{\mathcal I_{i-1}, \mathcal J_{i}}$.
        \ELSE
            \STATE {\bf "Vertical"} C--A step:
            \STATE 1. Let $C_i:=W_{:, \mathcal J_{i-1}}$ be an $m\times q$  submatrix of $W$.
            \STATE 2. Apply Sub-algorithm $\mathcal{A}$ for $p=m$ to $C_i$ and obtain a $k\times l$ submatrix $W_{i} = W_{\mathcal I_{i}, \mathcal J_{i-1}}$.
        \ENDIF
        \smallskip
        \STATE Apply Sub-algorithm $\mathcal B$ and obtain $E$, the error bound of CUR LRA built on the generator $W_i$.
        \IF{$E \le \tau$}
            \RETURN CUR LRA built on the generator $W_i$.
        \ENDIF
    \ENDFOR
    \RETURN {\bf Failure}
\end{algorithmic}
\end{quote}

\end{algorithm}

\bigskip
 
{\bf \Large PART II. CUR LRA FOR SPSD MATRICES}

\section{CUR LRA of SPSD Matrices: Main Results}\label{smnthr}


We are going to 
devise  an algorithm
to which we refer as the {\em Main Algorithm} and Algorithm \ref{alg:main_alg}
and and which computes  CUR LRA of an SPSD matrix at sublinear cost.


\begin{theorem}{{\bf (Main Result 1)}}\label{spsd:main_thm1}
 Suppose that $W\in\mathbb{R}^{n\times n}$ is an SPSD matrix, $r$ and $n$ are two positive integers, $r<n$,  $\epsilon$ is a  positive number, and $\mathcal{I}$ is the output of Algorithm \ref{alg:main_alg} (Main Algorithm). Write $C:=W_{:, \mathcal{I}}$, 
$U:=W_{\mathcal{I}, \mathcal{I}}^{-1}$, and $R:=W_{\mathcal{I}, :}$. Then
\begin{equation}
||W - CUR||_C \le (1+\epsilon)(r+1)\sigma_{r+1}(W).
\end{equation}
Furthermore Algorithm \ref{alg:main_alg} runs at an arithmetic  cost in $O(nr^4\epsilon^{-1}\log r)$
 as 
$\epsilon\rightarrow 0$.
\end{theorem}
\medskip

\begin{theorem}{{\bf (Main Result 2)
}}\label{spsd:main_thm2}
 Suppose that $W\in\mathbb{R}^{n\times n}$ is an SPSD matrix, $r$, $K$ and $n$ are three positive integers such that $r<K<n$,  $\epsilon$ is a a positive number, and $\mathcal{I}$ is the output of Algorithm \ref{alg:main_alg} (Main Algorithm). Write $C:=W_{:, \mathcal{I}}$, 
$U:=(W_{\mathcal{I}, \mathcal{I}})_r^{+}$, and $R:=W_{\mathcal{I}, :}$. Then
\begin{equation}
||W - CUR||_C \le (1+\epsilon)\frac{K+1}{K-r+1}\sigma_{r+1}(W).
\end{equation}
In particular, let $K = cr - 1$ for $c>1$. Then
\begin{equation}
||W - CUR||_C \le (1+\frac{1}{c-1})(1+\epsilon)\sigma_{r+1}(W).
\end{equation}
Furthermore Algorithm \ref{alg:main_alg} runs at an arithmetic  cost in 
$O((r+\log n)rK^4n\epsilon^{-1})$
 as 
$\epsilon\rightarrow 0$.

\end{theorem}

\section{Proof of Main Result 1}\label{sect:proof_thm1}

Hereafter $[n]$ denotes the set 
$\{1,2,\dots,n\}$ of $n$ integers $1,2,\dots,n$.

 $\mathcal{I} \subset [n]$ denotes a set of indices, while $|\mathcal I|$
denotes the cardinality (the number of elements) of  $\mathcal I$.

\begin{algorithm}
\caption{Greedy Column Subset Selection \cite{CM09}.}
\label{alg:greedy_column}
    \begin{quote}
    {\bf Input:} $A\in\mathbb{R}^{m\times 
n}$ and    a positive integer  $K < n$.\\
    {\bf Output:} $\mathcal{I}$.
    
    \begin{algorithmic}
    \STATE Initialize $\mathcal{I} = \{\}$. 
    \STATE $M^1 \leftarrow A$.
    \FOR{t = 1, 2, ..., K}
        \STATE $i \leftarrow \argmax_{a\in[n]} ||M^t_{:, a}||$
        \STATE $\mathcal{I} \leftarrow \mathcal{I}\cup \{i\}$.
        \STATE $M^{t+1} \leftarrow M^t - || M^t_{:, i} ||^{-2}  (M^t_{:, i}) ( M^t_{:, i})^T M^t  $
    \ENDFOR
    \RETURN $\mathcal{I}$.
    \end{algorithmic}\end{quote}
\end{algorithm}








Theorem \ref{thprncpl} below enables us to decrease the cost of searching for the maximal volume submatrix  an SPSD matrix by restricting the search to the principal submatrices.
As pointed out in \cite{CKM19} and 
implicit in \cite{CM09}, the task is still   NP hard and therefore is impractical for inputs of even moderately large size. 
\cite{CKM19} proposed to search for a submatrix with a large volume by means of a greedy algorithm that is essentially {\em Gaussian Elimination with Complete Pivoting} (Algorithm~\ref{alg:gaussian_pivoting} GECP). 
This, however, only guarantees an error bound of $4^r~\sigma_{r+1}(W)$ in the Chebyshev norm for the output  CUR LRA.

\begin{theorem}{(Cf. \cite{CKM19}.)}\label{thprncpl} 
Suppose that  $W$ is an $n\times n$ SPSD matrix and $\mathcal I$ and  $\mathcal J$ 
are two sets of integers in $[n]$  having the same cardinality. Then
$\volp{W_{\mathcal I,\mathcal J}}^2\le \volp{W_{\mathcal I,\mathcal I}}\volp{W_{\mathcal J,\mathcal J}}.$
\end{theorem}

\begin{algorithm}
\caption{An SPSD Matrix: Gaussian Elimination with Complete Pivoting (cf. 
\cite{CKM19}).} \label{alg:gaussian_pivoting}
\begin{quote}
    {\bf Input:} An SPSD matrix $A \in \mathbb{R}^{n\times n}$ and a positive integer $K < n$.\\
    {\bf Output:} $\mathcal{I}$.
    \begin{algorithmic}
    \STATE Initialize $R \leftarrow A$, and $\mathcal{I} = \{\}$.
    \FOR{t = 1, 2, ..., K}
        \STATE $i_t \leftarrow \argmax_{j\in [n]} |R_{j, j}|$.
        \STATE $\mathcal{I} \leftarrow \mathcal{I} \cup \{i_t\}$.
        \STATE $R \leftarrow R - R_{:, i_t} \cdot r_{i_t, i_t}^{-1} \cdot R_{i_t, :}$.
    \ENDFOR
    \RETURN $\mathcal{I}$.
    \end{algorithmic}
\end{quote}
\end{algorithm}



The papers \cite{GT01} and \cite{OZ18}
 enable us to simplify our task further by proving that if the generator $W_{\mathcal{I}, \mathcal{I}}$ has $(1+\epsilon)$-locally maximal volume, then the CUR LRA generated by it has error bound $(1+\epsilon)(1+r)\sigma_{r+1}(W)$ in 
Chebyshev norm.
This is a dramatic improvement over the error bound $4^r~\sigma_{r+1}(W)$ 
of \cite{CKM19}. 
For a general input matrix, finding  a submatrix with locally maximal volume or even verifying this property is still quite costly, because one would potentially need to compare the volumes of order of $n^2$ ``nearby''
 submatrices, where we say that two matrices are {\em nearby matrices} if they differ 
 at most
 by  a single row and/or  column.
 In the case of an SPSD input we 
 combine the results of \cite{GT01} and \cite{OZ18} 
 with   Theorem \ref{thprncpl}, and then
these combined results enable us to 
  verify
that a principal submatrix has maximal
or nearly  maximal
volume by
comparing its volume with that of its  $O(n)$ ``nearby" principal submatrices. 
We state this observation formally in the following Theorem \ref{thm:principle_locally_max}.


\begin{theorem}\label{thm:principle_locally_max}
Suppose that 
 $W\in \mathbb{R}^{n\times n}$ is an SPSD matrix, $\mathcal{I}$ is an index set, and $0<|\mathcal{I}| = r < n$. Let $(1+\epsilon)\cdot\volp{W_{\mathcal{I}, \mathcal{I}}} \ge \volp{W_{\mathcal{J}, \mathcal{J}}}$ for any index set $\mathcal{J}$ where $|\mathcal{J}| = r$, and $\mathcal{J}$ only differs from $\mathcal{I}$ at a single element.
Then  $W_{\mathcal{I}, \mathcal{I}}$ is a $(1+\epsilon)$-locally maximal volume submatrix. 
\end{theorem}


\begin{proof} 
Let $\mathcal{I}'$ and $\mathcal{J}'$ be any two index sets where $|\mathcal{I}'| = |\mathcal{J}'| = r$, and
let them 
differ from $\mathcal{I}$  at most
at a single element.
Apply \cite[Thm. 1]{CKM19} and obtain that $\volp{W_{\mathcal{I}', \mathcal{J}'}} \le \max \big( \volp{W_{\mathcal{I}', \mathcal{I}'}}, \volp{W_{\mathcal{J}', \mathcal{J}'}}  \big) \le (1+\epsilon)\cdot\volp{W_{\mathcal{I}, \mathcal{I}}}$.
The latter inequality follows from the assumption.
\end{proof}


Given an index set $\mathcal{I}_0$, with $|\mathcal{I}_0| = r$, and a 
positive
number $\epsilon$, Algorithm \ref{alg:index_Update} (Index Update) finds whether there is a ``nearby'' principal submatrices of $W_{\mathcal{I}_0, \mathcal{I}_0}$ whose volume is at least $(1+\epsilon)\cdot \volp{W_{\mathcal{I}_0, \mathcal{I}_0}}$, and if there is one, then updates the index set $\mathcal{I}_0$ accordingly. Therefore if we recursively apply this algorithm until no index update is possible, then the final index set $\mathcal{I}$ determines 
an $r\times r$ submatrix that has locally $(1+\epsilon)$-maximal volume; in this case the error bound of Main Result 1 can be readily deduced from Theorem \ref{th12}.  

Having proved 
correctness of Algorithm \ref{alg:main_alg} (Main Algorithm), we still need to 
estimate its
complexity. 
We only need
 to find 
 (i) a good initial index set $\mathcal{I}_0$  efficiently and 
(ii) a submatrix 
 having $(1+\epsilon)$-locally maximal volume. We achieve  this in  
  a relatively small number of recursive applications of Algorithm \ref{alg:index_Update} (Index Update). 

We first show that
Algorithm \ref{alg:gaussian_pivoting} (GECP) runs efficiently and determines a submatrix having 
a relatively large volume.  
For a general input matrix, the pivoting stage of GECP is especially costly because the element with the greatest modulus could appear anywhere in the residual matrix. For SPSD matrices, however, the residual is also SPSD; therefore we only need to examine such elements on the diagonal. 

Next we show that the initial principal submatrix found by Algorithm \ref{alg:gaussian_pivoting} (GECP) has a substantial volume.


\begin{theorem}\label{thm:volume_greedy}
(Adapted from \cite[Thm. 10]{CM09}.)
 Suppose that  Algorithm~\ref{alg:greedy_column} has been applied to an input 
matrix $C\in\mathbb{R}^{m\times n}$ and an integer
 $r$ from the set $[n]$ and that it has output a set $\mathcal{I}$. Then 
\begin{equation}
\volp{C_{:, \mathcal{I}}} \ge \frac{1}{r!} \max_{\mathcal{S} \subset [n]: |\mathcal{S}| = r} \volp{C_{:, \mathcal{S}}}.
\end{equation}
\end{theorem}


\begin{theorem}\label{thm:guassian_elemination_close_to_max}
For an SPSD matrix $W\in\mathbb{R}^{n\times n}$  and  a positive integer
$r < n$, let $\mathcal{I}$ be the output of Algorithm \ref{alg:gaussian_pivoting} (GECP) with inputs $W$ and $r$. Then
\begin{equation}
 \volp{W_{\mathcal{I}, \mathcal{I}}}\ge \frac{1}{(r!)^2} \max_{\mathcal{S}\subset [n]: |\mathcal{S}| = r} \volp{W_{\mathcal{S}, \mathcal{S}}}. 
\end{equation}
\end{theorem}

\begin{proof}
Since $W$ is SPSD, there exists $C \in\mathbb{R}^{n\times n}$ such that $W = C^TC$. Therefore, for any non-empty index set $\mathcal{J} \subset [n]$,
\begin{equation}
    W_{\mathcal{J}, \mathcal{J}} = C_{:, \mathcal{J}}^TC_{:, \mathcal{J}},
\end{equation}
and thus
\begin{equation}
    \volp{W_{\mathcal{J}, \mathcal{J}}} = \big(\volp{C_{:, \mathcal{J}}}\big)^2.
\end{equation}
Let us run Algorithm~\ref{alg:greedy_column} (Greedy Column Subset Selection) and Algorithm~\ref{alg:gaussian_pivoting} (GECP) with inputs ($C$, $k$) and ($W$, $k$), respectively, and let $\mathcal{I}$ and $\mathcal{I}'$ be the outputs, respectively. Since 
both algorithms use
greedy approach, we have
\begin{align}
    \volp{W_{\mathcal{I}, \mathcal{I}}} &= \big( \volp{C_{:, \mathcal{I}'}}\big)^2\\
                                       &\ge \frac{1}{(r!)^2} \max_{\mathcal{S} \subset [n]: |\mathcal{S}| = r} \volp{C_{:, \mathcal{S}}}^2\\
                                       &= \frac{1}{(r!)^2} \max_{\mathcal{S}\subset [n]: |\mathcal{S}| = r} \volp{W_{\mathcal{S}, \mathcal{S}}}.
\end{align}
\end{proof}

\bigskip
Let $\mathcal{I}_0$ be the initial index set obtained from Algorithm \ref{alg:gaussian_pivoting} (GECP), and let $\mathcal{I}_1, \mathcal{I}_2, \dots$ be
the
 index sets that we obtained by recursively applying Algorithm \ref{alg:index_Update} (Index Update). If we let $\epsilon > 0$, then the volumes of the corresponding submatrices are strictly increasing, and therefore 
  their sequence becomes invariant
  starting with some integer $T$. Furthermore, 
  $\volp{W_{\mathcal{I}_t, \mathcal{I}_t}}/\volp{W_{\mathcal{I}_0, \mathcal{I}_0}} \ge (1+\epsilon)^t$ for $t \le T$, and therefore 
  $T \le \log_{1+\epsilon} (r!)$.
Recall that $\log_{1+\epsilon} (a)=
\frac{\log (a)}{\log(1+\epsilon)}= 
O(\frac{\log (a)}{\epsilon})$as 
$\epsilon\rightarrow 0$, 
for all positive $a$.
For $a=r!$ this turns into equation
  $\log_{1+\epsilon} (r!)= O(r\epsilon^{-1}\log r)$ as 
$\epsilon\rightarrow 0$, 
and we arrive at the following result.

\begin{corollary}\label{thm:num_termination}
For an SPSD matrix $W\in\mathbb{R}^{n\times n}$ and an integer
$r$ in the set $[n-1]$,  Algorithm \ref{alg:main_alg} (Main Algorithm) calls Algorithm \ref{alg:index_Update} (Index Update)  $O(r \epsilon^{-1}\log r)$ times  as $\epsilon\rightarrow 0$.
\end{corollary}

\begin{algorithm}
\caption{Index Update}\label{alg:index_Update}
\begin{quote}
    {\bf Input:} An SPSD matrix $A\in\mathbb{R}^{n\times n}$, a set $\mathcal{I}\in [n]$, a positive integer $r \le |\mathcal{I}|$, and a positive number $\epsilon$.\\
    {\bf Output:} $\mathcal{J}$
    \begin{algorithmic}
        \STATE Compute $\volk\big(A_{\mathcal{I}, \mathcal{I}}\big)$
        \FORALL{$i \in \mathcal{I}$}
            \STATE $\mathcal{I}' \leftarrow \mathcal{I} - \{i\}$ 
            \FORALL{$j \in [n] - \mathcal{I}$ }
                \STATE $\mathcal{J} \leftarrow \mathcal{I}'\cup \{j\}$
                \STATE Compute $\volk\big(A_{\mathcal{J},\mathcal{J}}\big)$
                \IF{$\volk\big(A_{\mathcal{J},\mathcal{J}}\big)/\volk\big(A_{\mathcal{I}, \mathcal{I}}\big) > 1 + \epsilon$}
                    \RETURN $\mathcal{J}$
                \ENDIF
            \ENDFOR
        \ENDFOR
        \RETURN $\mathcal{I}$
    \end{algorithmic}
\end{quote}
\end{algorithm}

\begin{algorithm}
\caption{Main Algorithm}\label{alg:main_alg}
    \begin{quote}
        {\bf Input:} An SPSD matrix $A\in\mathbb{R}^{n\times n}$, two positive integers $K$ and $r$ such that $ r \le K < n$,  and a positive number $\epsilon$.
        (We let $K=r$ if $K$ is not specified.)\\
        {\bf Output:} $\mathcal{I}$
        \begin{algorithmic}
            \STATE $\mathcal{I} \leftarrow $ Algorithm \ref{alg:gaussian_pivoting}$(A, K)$
            \WHILE{{\bf TRUE}}
                \STATE $\mathcal{J} \leftarrow$ Algorithm \ref{alg:index_Update}$(A, \mathcal{I}, r, \epsilon)$
                \IF{$\mathcal{J} = \mathcal{I}$}
                    \STATE {\bf BREAK}
                \ELSE
                    \STATE $\mathcal{I} \leftarrow \mathcal{J}$
                \ENDIF
            \ENDWHILE
            \RETURN $\mathcal{I}$.
        \end{algorithmic}
    \end{quote} 
\end{algorithm}

\section{Proof of Main Result 2}\label{sect:proof_thm2}

We begin with
 an auxiliary result.





\begin{lemma}\label{lemma:proj_vol_max}
Suppose that $W\in\mathbb{R}^{n\times n}$ is an SPSD matrix, $\mathcal{I}$ and $\mathcal{J}$ are two non-empty index sets, $|\mathcal{I}| = |\mathcal{J}| = K < n$, and  $r \le K$ is a positive integer. Then
\begin{equation} 
    \textstyle{\volk \big(W_{\mathcal{I}, \mathcal{J}}\big)} \le \max \Big( \textstyle{\volk \big(W_{\mathcal{I}, \mathcal{I}}\big)}, \textstyle{\volk \big(W_{\mathcal{J}, 
    \mathcal{J}}\big)}    \Big).
\end{equation}
\end{lemma}

\begin{proof}
Let $W = C^TC$ for some $C \in \mathbb{R}^{n\times n}$. Then
\begin{align}
    \textstyle \volk \big( W_{\mathcal{I}, \mathcal{J}}\big) &= \textstyle\volk \big( C_{:, \mathcal{I}}^TC_{:, \mathcal{J}}\big) \\
                      &\le \textstyle\volk \big( C_{:, \mathcal{I}} \big)
                      \textstyle\volk \big( C_{:, \mathcal{J}}\big) \label{ineq:proj_vol}\\
                      &=\sqrt{\textstyle \volk \big( W_{\mathcal{I}, \mathcal{I}}\big)\textstyle \volk \big( W_{\mathcal{J}, \mathcal{J}}\big)}.
\end{align}
Inequality (\ref{ineq:proj_vol}) follows from Theorem \ref{thvolfctrg} (ii).
\end{proof}

Lemma \ref{lemma:proj_vol_max} shows that, 
similarly
to the case of 
matrix volume, the search for
the maximal $r$-projective volume submatrix of 
an SPSD matrix can 
be restricted to its
 principal submatrices and that 
Theorem \ref{thm:principle_locally_max} can be extended to submatrices with locally  maximal 
$r$-projective volume.
Next we formally state 
these observations

\begin{theorem}\label{thm:principle_locally_max_proj_vol}
Suppose that $W\in\mathbb{R}^{n\times n}$ is an SPSD matrix,  $\epsilon$ is a positive number, 
 $\mathcal{I} \subset [n]$ is an index set,  $|\mathcal{I}| = K$, $r$ is an integer,  $r < K < n$, and 
 for all $\mathcal{J} \subset [n]$, where $|\mathcal{J}| = K$ and $\mathcal{J}$ differs from $\mathcal{I}$ at a single element, the inequality $(1 + \epsilon)\cdot\volk\big(W_{\mathcal{I}, \mathcal{I}}\big) \ge \volk \big(W_{\mathcal{J}, \mathcal{J}}\big)$ holds. Then $W_{\mathcal{I}, \mathcal{I}}$ has $(1+\epsilon)$-locally maximal $r$-projective volume.
\end{theorem}

Based on the above result
 and extending the argument of Section \ref{sect:proof_thm1}  we can show that Algorithm \ref{alg:main_alg} (Main Algorithm) finds a principal submatrix with $(1+\epsilon)$-locally maximal $r$-projective volume. 
Together  with the result of Theorem \ref{th3} 
this implies 
correctness of 
 the error bound 
 of 
 {\bf Main 
 Result 2}, and we will bound the arithmetic cost next.

Let $T$ denote the number of 
invocations of Index Update required to arrive at a submatrix 
having $(1+\epsilon)$-locally maximal $r$-projective volume 
beginning from the initial principal submatrix obtained from Algorithm \ref{alg:gaussian_pivoting} (GECP). We 
deduce
 that 
 Algorithm \ref{alg:main_alg} (Main Algorithm) has sublinear complexity by 
proving that
 $T$ is sublinear in $n$. 
Recall the well-known bound 
on the volume of the greedily selected column submatrix from \cite{GE96},
and then the claim follows readily.


\begin{lemma} (Cf. \cite[Thm. 7.2]{GE96} and \cite[Thm. 10]{CM09}.)\label{lemma:init_proj_vol_close_to_max}
Suppose that $C\in\mathbb{R}^{m\times n}$, $r$ and $K$  are integers,
 $0<r \le K < n$, and $\mathcal{I}$ is the output of 
Algorithm \ref{alg:greedy_column} (Greedy Column Subset Selection) with input $C$ and $K$. Then

\begin{equation}
	\textstyle \volk \big( C_{:, \mathcal{I}}\big) 
	\ge 2^{-r(r-1)/2}n^{-r/2} 
	\textstyle \volk \big( C\big) 
\end{equation}
\end{lemma}

Theorem \ref{thm:guassian_elemination_close_to_max} implies that $\volk(W) / \volk (W_{\mathcal{I}, \mathcal{I}} ) \ge 2^{r(r-1)}n^{r}$, where $\mathcal{I}$ is the 
initial index set obtained by Algorithm \ref{alg:gaussian_pivoting} (GECP). Recall that $\volk(W)$ is the $r$-projective volume upper bound for any submatrix of $W$,
and therefore the number  $T$ of  required
 Index Updates
 is bounded by $\log_{1+\epsilon} (2^{r(r-1)}n^{r}) = O(r^2\epsilon^{-1} + \epsilon^{-1}r\log n)$.


\section{Complexity Analysis}\label{sect:comp_anal}
In this section, we estimate the time complexity of performing the Main Algorithm (Algorithm \ref{alg:main_alg}) 
 separately in two cases where $r = K$ and  $r < K$.
 
The cost of computing the initial set $\mathcal{I}_0$ by means of Algorithm \ref{alg:gaussian_pivoting} (GECP) is $O(nK^2)$. Let $t$ denote the number of iterations and let $c(r, K)$ denote the arithmetic cost of performing Algorithm \ref{alg:index_Update} (Index Update) with parameters $r$ and $K$. Then the complexity is in $O(nK^2 + t\cdot c(r, K))$.

In the case  where $r = K$, Corollary \ref{thm:num_termination} implies that $t = O(r\epsilon^{-1}\log r)$. 
 Algorithm \ref{alg:index_Update} (Index Update)  needs up to $nr$ comparisons of $\vol \big(W_{\mathcal{I}, \mathcal{I}}\big)$ and $\vol \big(W_{\mathcal{J}, \mathcal{J}}\big)$. Since $\mathcal{I}$ and $\mathcal{J}$ differs by at most at one index,  we compute $\vol \big(W_{\mathcal{J}, \mathcal{J}}\big)$ faster by using small rank update of 
 $\vol \big(W_{\mathcal{I}, \mathcal{I}})$ instead of computing from the scratch; this saves a factor of $k$. Therefore $c(r, r) = O(r^3n)$, and the time complexity of the Main Algorithm is 
 in $O(nr^4\epsilon^{-1}\log r)$. 

Finally let $r < K$. Then according to \cite[Theorem 7.2]{GE96} and \cite[Theorem 10]{CM09}, $t$ increases slightly -- staying 
in $O(r^2\epsilon^{-1} + \epsilon^{-1}r\log n)$; 
if $\volk \big(W_{\mathcal{J}, \mathcal{J}}\big)$ is computed by using SVD, then $c(r, K) = O(K^4n)$,
and the time complexity of the Main Algorithm is in $O((r+\log n)rK^4n\epsilon^{-1})$.


\smallskip 

\medskip

\medskip

{\bf \large PART III. CUR LRA BY MEANS OF C-A ITERATIONS}


\section{Overview}\label{svlmxmz}
  
   
In the next section  we show that 
already two successive C-A iterations output a CUR generator having $h$-maximal volume (for any $h>1$) if these iterations begin at 
a $p\times q$ submatrix of $W$ that shares its rank $r>0$ with $W$. By continuity of the volume the result is extended to perturbations of such matrices  within a norm bound estimated in Theorem \ref{thprtrbvlm}. 
   In Section \ref{sprdv_via_v}  we extend these results to the case where $r$-projective volume rather than the volume of  a   CUR generator is  maximized; Theorem \ref{th3} shows benefits of
such a maximization.
 In Section \ref{scurac} we
 estimate the complexity and accuracy  of a two-step loop of C--A iterations. In Sections \ref{sextpr}  and \ref{sHSS}
we demonstrate the power of C-A iterations by means of the acceleration of the FMM method to superfast level, that is, to performing it at sublinear cost.

   
\section{Volume of the Output of a C--A loop}\label{svlmcalp}


We can apply 
C--A steps  by choosing deterministic algorithms of \cite{GE96} for 
Sub-algorithm $\mathcal A$. In this case
  $mq$ and $pn$ memory cells and $O(mq^2)$
  and $O(p^2n)$ arithmetic operations are involved in
``vertical" and  ``horizontal" C-A iterations,
respectively. They run at sublinear cost if 
$p^2=o(m)$ and $q^2=o(n)$ and 
output submatrices having $h$-maximal volumes for $h$ being a low 
degree polynomial in $m+n$.
Every iteration outputs a matrix that has locally $h$-maximal volume in a ``vertical"
or ``horizontal" submatrix, and we hope to obtain globally $\bar h$-maximal submatrix
(for reasonably bounded $\bar h$) when maximization is performed recursively in alternate directions.

 
We readily prove that this expectation is true for a low rank input matrix and  hence, by virtue of Theorem \ref{thprtrbvlm},  for a matrix that admits its sufficiently close LRA.\footnote{In our proof we assume that a C-A step is not applied to $p\times n$ or $m\times  q$ input of volume 0. This is hard to ensure for the input families of our example in the Appendix, but realistically 
such a degeneracy is rare, and one can try to counter it by means of pre-processing 
$W\rightarrow FWH$ with random sparse orthogonal multipliers $F$ and $H$.} The error bound of the LRA deviates from optimal by the factor $\bar h$ even in  Chebyshev norm,  
    but in spite of this deficiency (see some remedy at the end of Section 1.1)
and a strong restriction on the input class, we still yield some limited
 formal support for the long-known  empirical efficiency of C-A iterations.

Let us next elaborate upon this support.

By comparing SVDs of the matrices $W$ and $W^+$  obtain the following lemma.

\begin{lemma}\label{le0} 
$\sigma_j(W)\sigma_{\rank (W)+1-j}(W^+)=1$ for 
 all matrices $W$ and all integers $j$ not exceeding 
 $\rank(W)$.
\end{lemma}

\begin{corollary}\label{co0} 
$v_2(W)v_2(W^+)=1$ 
and $v_{2,r}(W)v_{2,r}(W_r^+)=1$
for all  matrices $W$ of full rank and all integers $r$ such that $1\le r \le \rank(W)$.
\end{corollary}


We are ready to prove that a  $k\times l$ submatrix of rank $r$   has  $hh'$-maximal volume globally in 
 a rank-$r$
   matrix $W$, that is,
  over all  $k\times l$ submatrices of $W$, if it has
$(h,h')$ maximal  
nonzero volume in  this
   matrix $W$.

\begin{theorem}\label{th111} 
Suppose that the volume 
 of a $k\times l$ submatrix  
$W_{\mathcal I,\mathcal J}$ 
is  nonzero and
 $(h,h')$-maximal  in a  matrix $W$
  for $h\ge 1$ 
 and $h'\ge 1$
 where 
 $\rank(W)=r=\min\{k,l\}$.
 Then this volume is
$hh'$-maximal  
 over all its $k\times l$ submatrices of
 the matrix $W$.
\end{theorem}

\begin{proof} 
  The matrix 
  $W_{\mathcal I,\mathcal J}$ has full rank because its volume is nonzero.
  
 Fix any $k\times l$ submatrix 
 $W_{\mathcal I',\mathcal J'}$ of the matrix $W$, recall that $W=CUR$,  and
 obtain that 
$$W_{\mathcal I',\mathcal J'}=
W_{\mathcal I',\mathcal J}W_{\mathcal I,\mathcal J}^{+}W_{\mathcal I,\mathcal J'}.$$
 
If $k\le l$, then
first apply claim (iii) of Theorem \ref{thvolfctrg}  for $G:=W_{\mathcal I',\mathcal J}$  and $H:=W_{\mathcal I,\mathcal J}^{+}$;
 then apply claim (i)
of that theorem for $G:=W_{\mathcal I',\mathcal  J}W_{\mathcal I,\mathcal J}^{+}$
and $H:=W_{\mathcal I,\mathcal J'}$
and obtain that 
$$v_2(W_{\mathcal I',\mathcal J})=v_2(W_{\mathcal I',\mathcal J}W_{\mathcal I,\mathcal J}^{+}W_{\mathcal I,\mathcal J'})\le  
v_2(W_{\mathcal I',\mathcal J})v_2(W_{\mathcal I,\mathcal J}^{+})v_2(W_{\mathcal I,\mathcal J'}).$$
 
If $k>l$ deduce the same bound by applying
the same argument to the matrix equation
$$W_{\mathcal I',\mathcal J'}^T=
W_{\mathcal I,\mathcal J'}^TW_{\mathcal I,\mathcal J}^{+T}W_{\mathcal I',\mathcal J}^T.$$

Combine this bound with Corollary \ref{co0} for $W$ replaced by 
$W_{\mathcal I,\mathcal J}$ and deduce that
 \begin{equation}\label{eqloop1} 
 v_2(W_{\mathcal I',\mathcal J'})
 =v_2(W_{\mathcal I',\mathcal J}W_{\mathcal I,\mathcal J}^{+}W_{\mathcal I,\mathcal J'})\le v_2(W_{\mathcal I',\mathcal J})v_2 (W_{\mathcal I,\mathcal J'})/v_2(W_{\mathcal I,\mathcal J}).  
\end{equation}
Recall that the matrix $W_{\mathcal I,\mathcal J}$
is $(h,h')$-maximal
and conclude that
$$hv_2(W_{\mathcal I,\mathcal J})\ge v_2(W_{\mathcal I,\mathcal J'})~{\rm 
and}~ 
h'v_2(W_{\mathcal I,\mathcal J})\ge v_2(W_{\mathcal I',\mathcal J}).$$ 

Substitute these inequalities into the above bound on the volume $v_2(W_{\mathcal I',\mathcal J'})$ and obtain
that $v_2(W_{\mathcal I',\mathcal J'})\le hh' v_2(W_{\mathcal I,\mathcal J})$.
\end{proof}


\section{From the Maximal Volume to the Maximal $r$-Projective Volume}\label{sprdv_via_v}

  
Recall that the CUR LRA error bound of  Theorem \ref{th12} is strengthened when we shift to Theorem \ref{th3},
that is,  maximize 
$r$-projective volume for $r<k=l$ rather than the volume.
  Next 
 we reduce maximization of 
 $r$-projective  volume of a CUR
  generators  to volume
  maximization.

Recall that multiplication by a square orthogonal matrix changes neither the volume nor $r$-projective volume of a matrix and obtain the following result. 
 
 \begin{lemma}\label{leprwnmp} 
 Let $M$ and $N$ be a pair of $k\times l$
 submatrices of a $k\times n$ matrix
 and let $Q$ be a $k\times k$
 orthogonal matrix.
 Then 
 $v_2(M)/v_2(N)=v_2(QM)/v_2(QN)$,
and if $r\le \min\{k,l\}$ then
also 
$v_{2,r}(M)/v_{2,r}(N)=
v_{2,r}(QM)/v_{2,r}(QN)$.
\end{lemma} 


\begin{algorithm}[t]
\caption{From maximal volume to maximal $r$-projective volume}\label{algprjvlm}
\begin{quote}
    {\bf Input:} Four integers $k,l,n,$ and $r$ such that $0<r\le k\le n$
    and $r\le l \le n$; a $k\times n$ matrix $W$ of rank $r$; a black
    box algorithm that finds an $r\times l$ submatrix having locally
    maximal volume in an $r\times n$ matrix of full rank $r$.
    \\
    {\bf Output:} 
    A column set $\mathcal J$ such that $W_{:, \mathcal J}$ has
    maximal $r$-projective volume in $W$.
    \\ {\bf Computations:}
    \begin{algorithmic}
    \STATE 1. Compute a rank-revealing QRP factorization
    $W = QRP$, where $Q$ is an orthogonal matrix, $P$ is a permutation matrix,
    $R = \begin{pmatrix} R' \\ O \end{pmatrix}$, and 
    $R'$ is an $r\times n$ matrix.  (See \cite[Sections 5.4.3 and 5.4.4]{GL13} and \cite{GE96}.)
    \STATE 2. Compute an $r\times l$ submatrix $R'_{:, \mathcal J}$ of
    $R'$ having maximal volume.
    \RETURN $\mathcal J'$ such that $P: \mathcal J' \longrightarrow \mathcal J$. 
    \end{algorithmic}
\end{quote}
\end{algorithm}

The 
 submatrices $R'$ and
$\begin{pmatrix} R' \\O
\end{pmatrix}$ of the matrix $R$
of Algorithm \ref{algprjvlm} have  maximal volume and
 maximal 
$r$-projective volume in  the matrix $R$,
respectively, by virtue of Theorem \ref{thvolfctrg} and  
because $v_2(R)=v_{2,r}(R)=v_{2,r}(R')$. Therefore
the submatrix $W_{:,\mathcal J}$
has maximal $r$-projective volume in  the matrix $W$ by virtue of Lemma  \ref{leprwnmp}. 

 \begin{remark}\label{reprvtov} 
 By transposing a horizontal input matrix $W$ and interchanging the integers $m$ and $n$ and 
the integers $k$ and $l$ we extend the algorithm to computing
a $k\times l$ submatrix of maximal or nearly  maximal $r$-projective volume in  an $m\times l$ matrix of rank $r$.
\end{remark} 

   
\section{Complexity and Accuracy of an C--A Loop}\label{scurac}
    

The following theorem summarizes our study in this part of the paper.
 

\begin{theorem}\label{coc-avlm}
Given five integers $k$, 
 $l$,  $m$, $n$, and  $r$ such that
  $0<r\le k\le m$
 and  $r\le l\le n$, suppose that
  two successive C--A steps (say, based on the
  algorithms of 
\cite{GE96})
combined with Algorithm    
 \ref{algprjvlm}
  have been applied  
to an $m\times n$ matrix $W$ of rank $r$ and  have  output  a pair of
   $k\times l$ submatrices 
  $W'_{1}$ and   
  $W'_{2}=W_{\mathcal I_2,\mathcal J_2}$
   with nonzero $r$-projective
column-wise locally $h$-maximal and
nonzero $r$-projective row-wise locally  $h'$-maximal volumes, respectively.
Then  the submatrix $W_2'$
 has $h' h$-maximal $r$-projective volume in the matrix $W$.
\end{theorem}
 

 By combining
 Theorems \ref{th12},
\ref{th3}, and \ref{coc-avlm} we obtain the following corollary.  

\begin{corollary}\label{colclgl}
Under the assumptions of Theorem 
\ref{coc-avlm}
apply a two-step  C--A loop to an $m\times n$  matrix $W$ of rank $r$ 
and suppose that both its C--A steps 
output $k\times l$  submatrices having
nonzero $r$-projective column-wise and row-wise locally $h$-maximal volumes (see Remark \ref{reklprtr} below).
 Build
 a canonical CUR LRA  on
a CUR generator $W_2'=W_{k,l}$
of rank $r$ output by the second C--A step. 
Then 

(i) the computation of this CUR LRA by using 
the auxiliary algorithms of \cite{GE96} 
involves $(m+n)r$ 
memory cells and $O((m+n)r^2)$ arithmetic operations\footnote{For $r=1$ an input matrix
turns into a vector of dimension $m$ or $n$, and then we compute 
its absolutely maximal coordinate just by applying $m-1$
or $n-1$ comparisons, respectively (cf. \cite{O17}).}
and 

(ii) the error matrix $E$ of the output  
CUR LRA satisfies the bound
$||E||_C\le g(k,l,r)~\bar h~\sigma_{r+1}(W)$
for 
$\bar h$ of Theorem  
\ref{coc-avlm} and $g(k,l,r)$
denoting the functions $f(k,l)$
of Theorem \ref{th12}
or  $f(k,l,r)$
of Theorem \ref{th3}. In particular
$||E||_C\le 2hh' \sigma_{2}(W)$
for $k=l=r=1$.
\end{corollary}

 
\begin{remark}\label{reklprtr}
Theorem \ref{thprtrbvlm}  enables us to
 extend 
  Theorem \ref{coc-avlm} and 
Corollary \ref{colclgl} 
 to the case of an input matrix $W$ of
 numerical rank $r$ as long as a C-A step begins with a non-degenerating input matrix sharing its numerical rank $r$ with an input matrix $W$.
\end{remark}  



\section{Superfast Multipole Method}\label{sextpr} 
 

Sublinear cost LRA algorithms can be extended to numerous important computational problems linked to LRA.
Next we point out and
    extensively test a
    simple 
    application to the  acceleration of the celebrated Fast Multipole Method (FMM)
    (cf. \cite{GR87},  \cite{CGR88}, \cite{DGR96}, \cite{C00}, \cite{BY13}, and the bibliography therein),
    which  turns it into {\em Superfast Multipole Method}. 
 


%

FMM enables  multiplication  by a vector of so
called  HSS matrices at sublinear cost 
provided that low rank generators are available for its off-diagonal blocks.
Such generators can be a part of    the input, and it is frequently not emphasized that they are not available  in a variety of other  important applications of FMM (see, e.g., \cite{XXG12}, 
 \cite{XXCB14},  \cite{P15}).
 Even in such cases, however, by applying ACA  algorithms of \cite{B00},  \cite{B11} or the MaxVol algorithm of \cite{GOSTZ10} we can compute the generators 
at sublinear cost, thus turning FMM into Superfast 
 Multipole Method.  
 Next
we supply some details.                                                                                                                                                                                                               
 
 HSS matrices  
 naturally extend the class of banded matrices and their inverses, 
are closely linked to FMM,  
 and 
are increasingly popular.
See \cite{BGH03},  \cite{GH03}, \cite{MRT05},
 \cite{CGS07},
 \cite{VVGM05}, 
\cite{VVM07/08}, 
 \cite{B10}, \cite{X12},   \cite{XXG12},
\cite{EGH13}, \cite{X13},
 \cite{XXCB14},
  the references therein,
   software libraries 
H2Lib, http://www.h2lib.org/, 
https://github.com/H2Lib/H2Lib, and HLib, www.hlib.org, developed at the
Max Planck Institute for Mathematics in the Sciences. 

\begin{definition}\label{defneut} {\rm (Neutered Block Columns. See 
\cite{MRT05}.)} With each diagonal block of a block matrix 
associate 
its complement in its block column,
and call this complement a {\em neutered block column}.
\end{definition}

\begin{definition}\label{defqs} {\rm (HSS matrices. See 
\cite{CGS07},
 \cite{X12},  \cite{X13}, \cite{XXCB14}.)}
  
A block 
matrix $M$ of size  $m\times n$ is 
called an $r$-{\em HSS matrix}, for a positive integer $r$, 

(i) if all diagonal blocks of this matrix
consist of $O((m+n)r)$ entries overall
and
 
(ii) if $r$ is the maximal rank of its neutered block columns. 
\end{definition}

\begin{remark}\label{reqs}
Many authors work with $(l,u)$-HSS
(rather than $r$-HSS) matrices $M$ for which $l$ and $u$
are the maximal ranks of the sub- and super-diagonal blocks,
respectively.
The $(l,u)$-HSS and $r$-HSS matrices are closely related. 
If a neutered block column $N$
is the union of a sub-diagonal block $B_-$ and 
a super-diagonal block $B_+$,
then
 $\rank (N)\le \rank (B_-)+\rank (B_+)$,
 and so
an $(l,u)$-HSS matrix is an $r$-HSS matrix,
for $r\le l+u$,
while clearly an $r$-HSS matrix is  
a $(r,r)$-HSS matrix.
\end{remark}

The FMM exploits the $r$-HSS structure of a matrix as follows:

(i) Cover all off-block-diagonal entries
with a set of  non-overlapping neutered block columns.  

(ii) Express every neutered block column $N$ of this set
  as the product
  $FH$ of two  {\em generator
matrices}, $F$ of size $h\times r$
and $H$ of size $r\times k$. Call the 
pair $\{F,H\}$ a {\em length $r$ generator} of the 
neutered block column $N$. 

(iii)  Multiply 
the matrix $M$ by a vector by separately multiplying generators
and diagonal blocks by subvectors,  involving $O((m+n)r)$ flops
overall, and

(iv) in a more advanced application of  
 FMM solve a nonsingular $r$-HSS linear system of $n$
equations  by using
$O(nr\log^2(n))$ flops under some mild additional assumptions on  the input. 

This approach is readily extended to the same operations with  
$(r,\xi)$-{\em HSS matrices},
that is, matrices approximated by $r$-HSS matrices
within a perturbation norm bound $\xi$ where a  positive tolerance 
$\xi$ is small in context (for example, is the unit round-off).
 Likewise, one defines 
an  $(r,\xi)$-{\em HSS representation} and 
$(r,\xi)$-{\em generators}.

$(r,\xi)$-HSS matrices (for $r$ small in context)
appear routinely in matrix computations,
and computations with such matrices are 
performed  efficiently by using the
above techniques.


As we said already, in some applications of the FMM (see  \cite{BGP05} and  \cite{VVVF10})
stage (ii) is omitted because short generators for all 
neutered block columns are readily available,
 but this is not the case in some other important applications. 
This stage of the computation of  generators is precisely 
 LRA of the neutered block
columns; it turns out to be 
the bottleneck stage of FMM in these applications, and sublinear cost LRA algorithms  
provide a remedy. 

First suppose that a known LRA algorithm, e.g., the algorithm of \cite{HMT11}
 with a Gaussian multiplier, is applied  at this stage.
Multiplication of a $q\times h$ matrix
by an $h\times r$ Gaussian matrix requires $(2h-1)qr$ flops,
while
standard HSS-representation of an $n\times n$ 
HSS matrix includes $q\times h$ neutered 
 block columns for $q\approx m/2$ and $h\approx n/2$. In this case 
the cost of computing an $r$-HSS representation of the
matrix $M$ has at least order $mnr$.
For $r\ll \min\{m,n\}$, this
is much greater than 
 $O((m+n)r\log^2(n))$ flops, used  at the other stages of 
the computations. 
We, however, readily match  the latter bound   
at the stage of computing 
$(r,\xi)$-generators as well -- simply
by applying sublinear cost LRA algorithms.
    

\section{Computation of  LRAs for Benchmark HSS Matrices}
\label{sHSS}     
 

In this section we cover our   tests of the Superfast Multipole Method where we  applied C--A iterations
in order to compute  LRA of the generators of the off-diagonal blocks of 
HSS matrices.  Namely we tested HSS matrices that approximate $1024 \times 1024$ Cauchy-like matrices 
derived from  benchmark Toeplitz matrices B, C, D, E, and F of \cite[Section 5]{XXG12}. For the computation of  LRA we applied the algorithm of \cite{GOSTZ10}.

Table 1
displays the relative errors of
the approximation of the  $1024 \times 1024$ HSS input matrices
 in the spectral and +shev norms averaged over 100 tests.
 Each approximation was obtained by 
 means of   
 combining the exact diagonal blocks 
 and  LRA of the off-diagonal blocks.
  We
 computed  LRA of all these blocks 
 superfast.
 
In good accordance with extensive empirical evidence about the power of C--A iterations, already the first C--A loop has consistently output reasonably close  LRA, but our  further improvement was achieved in five C--A loops in our tests for all but one of the five families of input matrices.

The reported HSS rank is the larger of the numerical ranks for the $512 \times 512$ off-diagonal blocks. This HSS rank was used as an upper bound in our binary search that determined the numerical rank of each off-diagonal block for the purpose of computing its LRA. We based 
the binary search on  minimizing the difference (in the spectral norm) between each off-diagonal block and its LRA. 

The output error norms were quite low. Even in the case of
 the matrix C, obtained from Prolate Toeplitz matrices -- extremely ill-conditioned, they ranged from $10^{-3}$
 to $10^{-6}$.

We have also performed further numerical experiments on all the HSS input matrices by using a hybrid LRA algorithm: we used random pre-processing with Gaussian and Hadamard (abridged and permuted)  multipliers (cf. \cite{PLSZ19}) by incorporating them into
 Algorithm 4.1 of \cite{HMT11}, but applying it only to the  off-diagonal blocks of smaller sizes while retaining our previous 
 way for computing  LRA of the larger off-diagonal blocks. We have not displayed the results of these experiments because they yielded no substantial improvement in accuracy in comparison to the exclusive use of the less expensive  LRA on all off-diagonal blocks. 
 
\begin{table}[ht]
\begin{center}
\begin{tabular}{|c|c|c|c|c|c|c|}  
\hline
 &  & & \multicolumn{2}{c|}{\bf Spectral Norm} & \multicolumn{2}{c|}{\bf Chebyshev Norm} \\ \hline 

{\bf Inputs} & {\bf C--A loops}  &  {\bf HSS rank} & {\bf mean} & {\bf std} & {\bf mean} & {\bf std} \\ \hline


\multirow{2}*{B}

\multirow{2}*{B}
&1 &  26 & 8.11e-07 & 1.45e-06 & 3.19e-07 & 5.23e-07 \\  \cline{2-7}
&5 &  26 & 4.60e-08 & 6.43e-08 & 7.33e-09 & 1.22e-08 \\  \hline
\multirow{2}*{C}
&1 &  16 & 5.62e-03 & 8.99e-03 & 3.00e-03 & 4.37e-03 \\  \cline{2-7}
&5 &  16 & 3.37e-05 & 1.78e-05 & 8.77e-06 & 1.01e-05 \\  \hline
\multirow{2}*{D}
&1 &  13 & 1.12e-07 & 8.99e-08 & 1.35e-07 & 1.47e-07 \\  \cline{2-7}
&5 &  13 & 1.50e-07 & 1.82e-07 & 2.09e-07 & 2.29e-07 \\  \hline
\multirow{2}*{E}
&1 &  14 & 5.35e-04 & 6.14e-04 & 2.90e-04 & 3.51e-04 \\  \cline{2-7}
&5 &  14 & 1.90e-05 & 1.04e-05 & 5.49e-06 & 4.79e-06 \\  \hline
\multirow{2}*{F}
&1 &  37 & 1.14e-05 & 4.49e-05 & 6.02e-06 & 2.16e-05 \\  \cline{2-7}  
&5 &  37 & 4.92e-07 & 8.19e-07 & 1.12e-07 & 2.60e-07 \\  \hline
 
\end{tabular}

\caption{LRA approximation of HSS matrices from \cite{XXG12}}
\end{center}
\label{tabhss}
\end{table}


\bigskip 
\medskip


\smallskip 


\bigskip


{\bf \Large Appendix:}
{\bf Small families of hard inputs for sublinear  LRA}

\appendix 

\medskip


  Any sublinear cost LRA algorithm
  fails on the following  small input families.
  \smallskip
  
{\bf Example.} 
Define a family
of $m\times n$ matrices 
 of rank 1 (we call them $\delta$-{\em matrices}):  $$\{\Delta_{i,j}, ~{i=1, \dots,m;~j=1,\dots,n}\}.$$ 
 Also include the $m\times n$ null matrix $O_{m,n}$ into this family.
Now fix any sublinear cost  algorithm; it does not access the $(i,j)$th  
entry of its input matrices  for some pair of $i$ and $j$. Therefore it outputs the same approximation 
of the matrices $\Delta_{i,j}$ and $O_{m,n}$,
with an undetected  error at least 1/2.
Apply the same argument to the
set of $mn+1$ small-norm perturbations of 
the matrices of the above family and to the                                                                                                                                   
 $mn+1$ sums  
of the latter matrices with  any
   fixed $m\times n$ matrix of low rank.
Finally, the same argument shows that 
a posteriori estimation of the output errors of an LRA algorithm
applied to the same input families
cannot run at  sublinear cost. 

This example actually covers randomized LRA algorithms as well. Indeed suppose that with a positive constant probability  an LRA algorithm does not access $K$  entries of an input matrix and apply 
this algorithm to two matrices of low rank whose  difference at all these entries is equal to a large constant $C$. Then, clearly, with a positive constant probability the algorithm has errors
at least $\frac{C}{2}$  at at least $\frac{K}{2}$
of these entries. 


\medskip 


\noindent {\bf Acknowledgements:} 
Our research has been supported by NSF Grants CCF--1116736, CCF--1563942,  CCF--133834
and PSC CUNY Awards 62797-00-50 and 63677-00-51.
We also thank A. Cortinovis,
A. Osinsky, N. L. Zamarashkin
for pointers to their papers \cite{CKM19}
 and \cite{OZ18}, 
S. A. Goreinov for sending reprints  of his papers,
 and E. E. Tyrtyshnikov for 
pointers to the bibliography and
  the challenge of
 formally supporting empirical power of C--A 
 algorithms.



\end{document}